\documentclass[11pt,twoside]{amsart}
\usepackage{amsmath, amsthm, amscd, amsfonts, amssymb, graphicx, color}
\usepackage[bookmarksnumbered, colorlinks, plainpages]{hyperref}

\newtheorem{thm}{Theorem}[section]
\newtheorem{cor}[thm]{Corollary}

\newtheorem{defn}[thm]{Definition}
\newtheorem{rem}[thm]{\bf{Remark}}
\newtheorem{alg}{\bf{Algorithm}}

\numberwithin{equation}{section}

\begin{document}


\title{System of split variational inequality problems}
\author{Kaleem Raza Kazmi}

\thanks{{\scriptsize
\hskip -0.4 true cm MSC(2010): Primary: 47J53; Secondary: 90C25.
}}

\maketitle

\begin{center}
{\it Department of
Mathematics, Aligarh Muslim University \\Aligarh 202002, India}\\

{ Email: krkazmi@gmail.com}
\end{center}

\begin{abstract}   In this paper, we introduce a system of split variational inequality problems in  real Hilbert spaces. Using projection method, we propose an iterative algorithm for the system of split variational inequality problems. Further, we prove that the  sequence generated by the   iterative algorithm converges  strongly to a   solution of the system of split variational inequality problems. Furthermore, we discuss some consequences of the main result. The iterative algorithms and results presented in this paper   generalize, unify and improve the previously known results of this area.
\end{abstract}

\vskip 0.2 true cm


\pagestyle{myheadings}
\markboth{\rightline {\scriptsize  Kazmi}}
         {\leftline{\scriptsize System of split variational inequality problems}}

\bigskip
\bigskip


\section{\bf Introduction}
\vskip 0.4 true cm

Throughout the paper unless otherwise stated, for each $s\in\{1,2,3,4\}$, let $H_s$  be a real Hilbert space  with inner product $\langle \cdot,\cdot \rangle $ and norm $\| \cdot \|$; let $C_s$  be a nonempty, closed and convex subset of $H_s$.

\vspace{.3cm}
The {\it variational inequality problem} (in short, VIP) is to find $x\in C_1$  such that
$$\langle h_1(x),y-x\rangle\geq 0,~~\forall y\in C_1, \leqno(1.1)$$
where $h_1:C_1\to H_1$ be a nonlinear mapping.

\vspace{.3cm}

Variational inequality theory introduced by Stampacchia {\cite{stam}} and Fichera {\cite{fich}} independently, in early sixties in potential theory and mechanics, respectively, constitutes a significant extension of variational principles. It has been shown that the variational inequality theory provides the natural, descent, unified and efficient framework for a general treatment of a wide class of unrelated linear and nonlinear problem arising in elasticity, economics, transportation, optimization, control theory and engineering sciences {\cite{bai,ben,crank,gian,glo,kiku}}. The development of variational inequality theory can be viewed as the simultaneous pursuit of two different lines of research. On the one hand, it reveals the fundamental facts on the qualitative behavior of solutions to important classes of problems. On the other hand, it enables us to develop highly efficient and powerful numerical methods to solve, for example, obstacle, unilateral, free and moving boundary value problems. In last five decades, considerable interest has been shown in developing various classes of variational inequality problems, both for its own sake and for its applications.\\

\vspace{.30cm}

In 1985, Pang {\cite{pang}} showed that a variety of equilibrium models, for example, the traffic equilibrium problem, the spatial equilibrium problem, the Nash equilibrium problem and the general equilibrium programming problem can be uniformly modelled as a variational inequality problem defined on the product sets. He decomposed the original variational inequality problem into a system of variational inequality problems and discuss the convergence of method of decomposition for system of variational inequality problems. Later, it was noticed that variational inequality problem over product sets and the system of variational inequality problems both are equivalent, see for applications {\cite{fer,nag,pang}}. Since then many authors, see for example {\cite{coh,fer,kas,nag}} studied the existence theory of various classes of system of variational inequality problems by exploiting fixed-point theorems and minimax theorems. On the other hand, a number of iterative algorithms have been constructed for approximating the solution of systems of variational inequality  problems, see {\cite{kaz8,kaz7,kaz6,kon,ver}} and the relevant references therein.

\vspace{.30cm}
 More precisely, the system of variational inequality problems (in short, SVIP) is to find $(x,y)\in C_1\times C_2$ such that
$$
\langle F(x,y), z_1-x\rangle\geq 0,~~\forall z_1\in C_1, \leqno(1.2)$$
$$\langle G(x,y), z_2-x\rangle\geq 0,~~\forall z_2\in C_2, \leqno(1.3)$$
where $F:C_1 \times C_2 \to H_1$ and $G :C_1 \times C_2 \to H_2$.  Verma {\cite{ver}}  studied the convergence analysis of an iterative method for a problem similar to SVIP(1.2)-(1.3) by using projection mappings.

\vspace{.30cm}

Recently, Censor {\it et al.} {\cite{cen3}} introduced   the following split variational inequality problem (in short, SpVIP): Find $x\in C_1$ such that
$$\langle h_1(x),z_1-x\rangle\geq 0,~~\forall z_1 \in C_1,\leqno(1.4)$$
and such that
$$y=Ax\in C_2~~{\rm solves}~\langle h_2(y),z_2-y\rangle\geq 0,~~\forall z_2 \in C_2,\leqno(1.5)$$
where $h_1:H_1\to H_1$ and $h_2:H_2\to H_2$ are nonlinear mappings  and $A:H_1\to H_2$ is a bounded linear operator. They studied some iterative methods for SpVIP(1.4)-(1.5).

\vspace{.3cm}
 SpVIP(1.4)-(1.5) is an important generalization of VIP(1.1). It also includes as special case, the split zero problem  and split feasibility problem  which has already been studied and used in practice as a model in intensity-modulated radiation therapy  treatment planning, see {\cite{cen2,cen1}}. For the further related work, we refer to see Moudafi {\cite{mou}},  Byrne {\it et al.} {\cite{byr}}, Kazmi and Rizvi {\cite{kaz5,kaz4,kaz3,kaz2}} and Kazmi {\cite{kaz1}}.

 \vspace{.3cm}
 Motivated by the work of Censor {\it et al.} {\cite{cen3}}, Kazmi {\cite{kaz1}}, Verma {\cite{ver}} and work going in this direction,  we introduce  the following  system of split variational inequality problems (in short, SSpVIP), which is a natural generalization of SpVIP(1.4)-(1.5):

\vspace{.3cm}

 Let $F:C_1\times C_2 \to H_1$, $G:C_1\times C_2 \to H_2$, $f:C_3\times C_4 \to H_3$ and $g:C_3\times C_4 \to H_4$ be nonlinear bifunctions and $A:H_1\to H_3$ and $B:H_2\to H_4$ be  bounded linear operators, then SSpVIP is to  find $(x,y) \in C_1 \times C_2$ such that
$$ \langle F(x,y), z_1-x\rangle\geq 0,~~\forall z_1\in C_1, \leqno(1.6)$$
and such that $(u,v)$ with $u=Ax \in C_3, ~~v=By \in C_4$ solves
$$\langle f(u,v), z_3-u\rangle\geq 0,~~\forall z_3\in C_3; \leqno(1.7)$$
$$ \langle G(x,y), z_2-y\rangle\geq 0,~~\forall z_2\in C_2, \leqno(1.8)$$
 and such that $(u,v)$ solves
$$\langle g(u,v), z_4-v\rangle\geq 0,~~\forall z_4\in C_4. \leqno(1.9)$$

\newpage
\noindent {\it Some special cases:}\\

\noindent {\bf I.}~~~If we set $H_2=H_1,~H_4=H_3,~C_2=C_1,~C_4=C_3,~G=F,~g=f,~B=A,$ and $y=x$,  then SSpVIP(1.6)-(1.9) reduces to the following spilt variational inequality problem: Find $x \in C_1$ such that
$$ \langle F(x,x), z_1-x\rangle\geq 0,~~\forall z_1\in C_1, \leqno(1.10)$$
and such that $u=Ax \in C_3$ solves
$$\langle f(u,u), z_3-u\rangle\geq 0,~~\forall z_3\in C_3. \leqno(1.11)$$
 The SpVIP(1.10)-(1.11) is new and different from SpVIP(1.4)-(1.5).

\vspace{.3cm}
\noindent {\bf II.}~~~If we set $H_3=H_1,~H_4=H_2,~C_3=C_1,~C_4=C_2,~f=F,~g=G,$ and $A=B=I$, identity mapping, then SSpVIP(1.6)-(1.9) reduces to the SVIP(1.2)-(1.3).

\vspace{.3cm}
 Using projection method, we propose an iterative algorithm for SSpVIP (1.6)-(1.9) and discuss some of its special cases.   Further, we prove that the  sequence generated by the   iterative algorithm converges  strongly to a   solution of SSpVIP(1.6)-(1.9). Furthermore, we discuss some consequences of the main result. The iterative algorithms and results presented in this paper   generalize, unify and improve the previously known results of this area, see for example {\cite{kaz1,ver}}.\\


\section{\bf {\bf \em{\bf Iterative Algorithms}}}
\vskip 0.4 true cm

For each $s\in\{1,2,3,4\}$, a mapping $P_{C_s}$ is said to be  {\it metric projection} of $H_s$ onto $C_s$ if for every point  $x_s \in H_s$, there exists a unique nearest point in $C_s$ denoted by $P_ {C_s} (x_s)$ such that
$$ \|x_s-P_{C_s}(x_s)\|\leq \|x_s-y_s\|, ~~ \forall  y_s \in C_s.$$
 It is well known that $P_{C_s}$ is nonexpansive mapping and satisfies
$$\langle x_s-y_s ,P_{C_s}(x_s)-P_{C_s}(y_s) \rangle \geq \|P_{C_s}(x_s)-P_{C_s}(y_s)\|^2, ~~\forall x_s,y_s \in H_s.\leqno(2.1)$$
Moreover, $P_{C_s}(x_s)$ is characterized by the following properties:
$$\langle x_s-P_{C_s}(x_s),y_s-P_{C_s}(x_s) \rangle \leq 0,\leqno(2.2)$$
and
$$\|x_s-y_s\|^2\geq \|x_s-P_{C_s}(x_s)\|^2 +\|y_s-P_{C_s}(x_s)\|^2,~~\forall x_s\in H_s,~y_s\in C_s.\leqno(2.3)$$

 \vspace{.3cm}
Further, it is easy to see that the following is true:
$$x {\rm ~is~a~solution~of~ VIP(1.1)}\Leftrightarrow x=P_{C_1} (x-\lambda h_1(x)),~~\lambda>0.\leqno(2.4)$$

Hence, SSpVIP(1.6)-(1.9) can be reformulated as follows:~~ Find $(x,y) \in C_1 \times C_2$ with $(u,v)=(A(x),B(y)) \in C_3 \times C_4$ such that
$$x= P_{C_1}(x- \rho F(x,y)),$$
$$u= P_{C_3}(u- \lambda f(u,v)),$$
$$y= P_{C_2}(y- \rho G(x,y)),$$
$$v= P_{C_4}(v- \lambda g(u,v)),$$
for $\rho,~\lambda >0$.

\vspace{0.3cm}
Based on above arguments, we propose the following iterative algorithm for approximating a solution to SSpVIP(1.6)-(1.9).

\vspace{0.3cm}
Let $\{\alpha_n\} \subseteq (0,1)$ be a sequence such that $\sum \limits^{\infty}_{n=1} \alpha_n=+\infty$, and let $\rho,~ \lambda,~ \gamma$ are parameters with positive values.

\vspace{0.3cm}
\begin{alg}
 Given $(x_0,y_0)\in C_1 \times C_2,$ compute the iterative sequence $\{(x_n, y_n)\}$ defined by the iterative schemes:
$$a_n= P_{C_1}(x_n- \rho F(x_n,y_n)),\leqno(2.5)$$
$$d_n= P_{C_2}(y_n- \rho G(x_n,y_n)),\leqno(2.6)$$
$$b_n= P_{C_3}(A(a_n)- \lambda f(A(a_n),B(d_n))), \leqno(2.7)$$
$$l_n= P_{C_4}(B(d_n)- \lambda g(A(a_n),B(d_n))),\leqno(2.8)$$
$$x_{n+1}=(1-\alpha_n)x_n +\alpha_n[a_n+\gamma A^*(b_n-A(a_n))] \leqno(2.9)$$
$$y_{n+1}=(1-\alpha_n)y_n +\alpha_n[d_n+\gamma B^*(l_n-B(d_n))] \leqno(2.10)$$
for all $n=0,1,2,.....~$ and $\rho,~\lambda,~\gamma >0$, where $A^*$ and $B^*$ are, respectively, the adjoint operator of $A$ and $B$.
\end{alg}

\noindent{\it Some special cases:}\\

If we set $H_3=H_1,~H_4=H_2,~C_3=C_1,~C_4=C_2,~f=F,~g=G,$ and $A=B=I$, identity mapping, then Algorithm 1 reduces to the following iterative algorithm for SVIP(1.2)-(1.3).

\vspace{.3cm}
\begin{alg} Given $(x_0,y_0)\in C_1 \times C_2,$ compute the iterative sequence $\{(x_n, y_n)\}$ defined by the iterative schemes:
$$a_n= P_{C_1}(x_n- \rho F(x_n,y_n)),$$
$$d_n= P_{C_2}(y_n- \rho G(x_n,y_n)),$$
$$x_{n+1}=(1-\alpha_n)x_n +\alpha_na_n $$
$$y_{n+1}=(1-\alpha_n)y_n +\alpha_nd_n $$
for all $n=0,1,2,.....,$ and  $\rho >0$.
\end{alg}

\vspace{.3cm}
If we set $H_2=H_1,~H_4=H_3,~C_2=C_1,~C_4=C_3,~G=F,~g=f,~B=A,$ and $y=x$,  then Algorithm 1 reduces to the following iterative algorithm for SpVIP(1.10)-(1.11).

\vspace{0.3cm}
\begin{alg} Given $x_0\in C_1,$ compute the iterative sequence $\{x_n\}$ defined by the iterative schemes:
$$a_n= P_{C_1}(x_n- \rho F(x_n,x_n)),$$
$$b_n= P_{C_3}(A(a_n)- \lambda f(A(a_n),A(a_n))),$$
$$x_{n+1}=(1-\alpha_n)x_n +\alpha_n[a_n+\gamma A^*(b_n-A(a_n))]$$
for all $n=0,1,2,.....~$ and $\rho,~\lambda,~\gamma >0$, where $A^*$ is the adjoint operator of $A$.
\end{alg}

\vspace{.3cm}
\begin{defn} A mapping $F:H_1\times H_2 \rightarrow H_1$ is said to be
\begin{enumerate}
\item [{(i)}] $\alpha_1$-{\it strongly monotone} in the first argument, if there exists a constant $\alpha_1>0$ such that
$$\langle F(x_1,y)-F(x_2,y),x_1-x_2\rangle \geq \alpha_1 \Vert x_1-x_2 \Vert^2,~~~ \forall x_1,x_2\in H_1, y\in H_2;$$
\item [{(ii)}] $\alpha$-{\it strongly monotone} in the second argument, if there exists  a constant  $\alpha>0$ such that
$$\langle F(x,y_1)-F(x,y_2),y_1-y_2\rangle \geq \alpha \Vert y_1-y_2 \Vert^2,~~~ \forall  x\in H_1,~y_1,y_2\in H_2;$$
\item [{(iii)}]  $(\beta_1,\epsilon_1)$-{\it Lipschitz continuous}, if there exist constants  $\beta_1 >0, \epsilon_1 >0$ such that
$$\|F(x_1,y_1)-F(x_2,y_2)\|\leq \beta_1 \|x_1-x_2\|+\epsilon_1 \|y_1-y_2\|,~~~ \forall x_1,x_2\in H_1,~ y_1,y_2\in H_2.$$
\end{enumerate}
\end{defn}

\vspace{.3cm}
\begin{defn} A mapping $S:H_1\times H_1 \rightarrow H_1$ is said to be
\begin{enumerate}
\item [{(i)}] $\eta$-{\it strongly mixed monotone}, if there exists a constant $\eta>0$ such that
$$\langle S(x_1,x_1)-S(x,x),x_1-x\rangle \geq \eta\Vert x_1-x \Vert^2,~~~ \forall x_1,x\in H_1;$$
\item [{(ii)}]  $\xi$-{\it mixed Lipschitz continuous}, if there exists a constant  $\xi >0$ such that
$$\|S(x_1,x_1)-S(x,x)\|\leq \xi \|x_1-x\|,~~~ \forall x_1,x\in H_1.$$
\end{enumerate}
\end{defn}

\section{\bf {\bf \em{\bf Results}}}
\vskip 0.4 true cm

\vspace{.3cm}
Now, we prove that the iterative sequence generated by Algorithm 1 converges strongly to a solution of SSpVIP(1.6)-(1.9).

\vspace{.3cm}
\begin{thm} \label{main} For each  $s\in\{1,2,3,4\},$ let $C_{s}$ be a nonempty, closed and convex subset of real Hilbert space $H_{s}$;  let $F: H_1\times H_2\to H_1$ be $\alpha_{1}$-strongly monotone in the first argument and $(\beta_{1},\epsilon_{1})$-Lipschitz continuous;
  let $G: H_1\times H_2\to H_2$ be $\alpha_{2}$-strongly monotone in the second argument and $(\epsilon_{2},\beta_{2})$-Lipschitz continuous;
let $f: H_3\times H_4\to H_3$ be $\sigma_{1}$-strongly monotone in the first argument and $(\mu_{1},\nu_{1})$-Lipschitz continuous, and
let $g: H_3\times H_4\to H_4$ be $\sigma_{2}$-strongly monotone in the second argument and $(\nu_{2},\mu_{2})$-Lipschitz continuous. Let $A:H_1 \to H_3$ and $B:H_2 \to H_4$ be bounded linear operators. Suppose $(x,y) \in C_1 \times C_2$ is a solution to SSpVIP(1.6)-(1.9) then the sequence $\{(x_n, y_n)\}$ generated by Iterative algorithm 1  converges strongly to $(x,y)$ provided that  for $i\in \{1,2\},~~j \in \{1,2\}\backslash \{i\}$, the constants $\rho, \lambda, \gamma$  satisfy the conditions:
$$\max_{1\leq i \leq 2}\{L_i-\Delta_i\}<\rho < \min_{1\leq i \leq 2}\{L_i+\Delta_i\}$$
$$ L_i=\frac{\alpha_i-p_iq_i}{\beta_i^2 -p_i^2};~~~\Delta_i=\frac{\sqrt{(\alpha_i-p_iq_i)^2-(\beta_i^2 -p_i^2)(1-q_i^2)}}{\beta_i^2 -p_i^2};$$
$$\alpha_i>p_iq_i+\sqrt{(\beta_i^2 -p_i^2)(1-q_i^2)};~~\beta_i > p_i; q_i<1$$
$$p_i=\frac{(\delta_j+2 \lambda \nu_i)\epsilon_j}{\delta_i+2 \lambda \nu_j};~~~q_i=\frac{1}{\delta_i+2 \lambda \nu_j}; ~~\delta_i=(1+2\theta_{i+2});$$
$$\theta_{i+2}=\sqrt{1-2\lambda\sigma_i+\lambda^2\mu_i^2};~~ \lambda > 0;~~ \gamma \in \left(0, \min\left\{\frac{2}{\|A\|^2},\frac{2}{\|B\|^2}\right\}\right).$$
\end{thm}

\begin{proof} Given that $(x,y)$ is a solution of SSpVIP(1.6)-(1.9), that is, $x,y$ satisfy the following relations:
$$x= P_{C_1}(x- \rho F(x,y)),\leqno(3.1)$$
$$y= P_{C_2}(y- \rho G(x,y)), \leqno(3.2)$$
$$A(x)= P_{C_3}(A(x)- \lambda f(A(x),B(y))),\leqno(3.3)$$
$$B(y)= P_{C_4}(B(y)- \lambda g(A(x),B(y))). \leqno(3.4)$$

Since $F: H_1\times H_2\to H_1$ be $\alpha_{1}$-strongly monotone in the first argument and $(\beta_{1},\epsilon_{1})$-Lipschitz continuous, from Algorithm 1(2.5) and (3.1), we estimate
$$\|a_n-x\| = \|P_{C_1}(x_n- \rho F(x_n,y_n))-P_{C_1}(x- \rho F(x,y))\| \hspace{2in}$$
$$ \leq \|x_n-x- \rho( F(x_n,y_n)-F(x,y_n))\|+\rho \|F(x,y_n)-F(x,y)\|\hspace{1in}$$
$$\leq  \left(\|x_n-x\|^2- 2\rho \langle F(x_n,y_n)-F(x,y_n), x_n-x\rangle \right.\hspace{2in}$$
$$\left. +\rho^2\|F(x_n,y_n)-F(x,y_n)\|^2\right)^{\frac{1}{2}}+\rho \|F(x,y_n)-F(x,y)\|\hspace{1in}$$
$$\leq \theta_1\|x_n-x\| + \rho \epsilon_1\|y_n-y\|,\hspace{3in}$$ $$~~ \leqno(3.5)$$
where $\theta_1=\sqrt{1-2\rho \alpha_1+\rho^2\beta_1^2}.$

\vspace{.3cm}
Next, since $G: H_1\times H_2\to H_2$ be $\alpha_{2}$-strongly monotone in the second argument and $(\beta_{2},\epsilon_{2})$-Lipschitz continuous, from Algorithm 1(2.6) and (3.2), we have
$$\|d_n-y\|  \leq \theta_2\|y_n-y\| + \rho \epsilon_2\|x_n-x\|, \leqno(3.6)$$
where $\theta_2=\sqrt{1-2\rho \alpha_2+\rho^2\beta_2^2}.$

\vspace{.3cm}
Again, since $f: H_3\times H_4\to H_3$ be $\sigma_{1}$-strongly monotone in the first argument and $(\mu_{1},\nu_{1})$-Lipschitz continuous,  from Algorithm 1(2.7) and (3.3), we have
$$\|b_n-A(x)\|  \leq \theta_3\|A(a_n)-A(x)\| + \lambda \nu_1\|B(d_n)-B(y)\|, \leqno(3.7)$$
where $\theta_3=\sqrt{1-2\lambda \sigma_1+\lambda^2\mu_1^2}.$

\vspace{.3cm}
Since $g: H_3\times H_4\to H_4$ be $\sigma_{2}$-strongly monotone in the second argument and $(\mu_{2},\nu_{2})$-Lipschitz continuous,  from Algorithm 1(2.8) and (3.4), we have
$$\|l_n-B(y)\|  \leq \theta_4\|B(d_n)-B(y)\| + \lambda \nu_2\|A(a_n)-A(x)\|, \leqno(3.8)$$
where $\theta_4=\sqrt{1-2\lambda \sigma_2+\lambda^2\mu_2^2}.$

\vspace{.3cm}
Now, using the definition of $A^*$, fact that $A^*$ is a bounded linear operator with $\|A^*\|=\|A\|$, and condition $\gamma \in \left(0, \min \left\{\frac{2}{\|A\|^2},\frac{2}{\|B\|^2}\right \}\right)$, we have
$$\|a_n-x-\gamma A^*(A(a_n)-A(x))\|^2\hspace{3.5in}$$
$$=\|a_n-x\|^2-2\gamma \langle a_n-x,A^*(A(a_n)-A(x))\rangle + \gamma^2\|A^*(A(a_n)-A(x))\|^2$$
$$\hspace{-1.8in}\leq\|a_n-x\|^2- \gamma(2- \gamma\|A\|^2) \|A(a_n)-A(x)\|^2 $$
$$\leq \|a_n-x\|^2.\hspace{3.9in}$$ $$~~\leqno(3.9)$$

Similarly, using the definition of $B^*$, fact that $B^*$ is a bounded linear operator with $\|B^*\|=\|B\|$, and condition $\gamma \in \left(0, \min \left\{\frac{2}{\|A\|^2},\frac{2}{\|B\|^2}\right\}\right)$, we have
$$\|d_n-y-\gamma B^*(B(d_n)-B(y))\|\leq \|d_n-y\|.\leqno(3.10)$$

From Algorithm 1(2.9),(3.5),(3.7) and (3.9), we have the following estimate:
$$\|x_{n+1}-x\|\hspace{4in}$$
$$\leq (1-\alpha_n)\|x_{n}-x\|\hspace{4in}$$
$$+\alpha_n \left[\|a_{n}-x-\gamma A^*(A(a_n)-A(x))\| +\gamma\|A\|\|b_n-A(x)\|\right]\hspace{3in}$$
$$\leq (1-\alpha_n)\|x_{n}-x\|\hspace{4in}$$
$$+\alpha_n \left[\|a_{n}-x\|+\gamma \|A\|[\theta_3\|A(a_n)-A(x)\|+\lambda\nu_1\|B(d_n)-B(y)\|]\right]$$
$$\leq (1-\alpha_n)\|x_{n}-x\|\hspace{4in}$$
$$+\alpha_n \left[(1+\gamma \|A\|^2\theta_3)\|a_{n}-x\|+\lambda\nu_1\|A\|\|B\|\|d_n-y\|\right]\hspace{1in}$$
$$\leq (1-\alpha_n)\|x_{n}-x\|\hspace{4in}$$
$$+\alpha_n \left[\delta_1\theta_1\|x_{n}-x\|+\rho\delta_1\epsilon_1\|y_n-y\| +\lambda\nu_1\gamma\|A\|\|B\|\theta_2\|y_n-y\| \right.\hspace{1in}$$
$$\left. +\lambda \rho\nu_1\epsilon_2\gamma\|A\|\|B\|\|x_n-x\|\right],\hspace{3in}$$ $$~~\leqno(3.11)$$

\noindent where $\delta_1=(1+\gamma \|A\|^2\theta_3)$.

\vspace{.3cm}
From Algorithm 1(2.10),(3.6),(3.8) and (3.10), we have the following estimate:
$$\|y_{n+1}-y\|\hspace{4in}$$
$$\leq (1-\alpha_n)\|y_{n}-y\|\hspace{4in}$$
$$+\alpha_n \left[\|d_{n}-y-\gamma B^*(B(d_n)-B(y))\|+\gamma\|B\|\|l_n-B(y)\|\right]$$
$$\leq (1-\alpha_n)\|y_{n}-y\|\hspace{4in}$$
$$+\alpha_n \left[(1+\gamma \|B\|^2\theta_4)\|d_{n}-y\|+\lambda\nu_2\gamma\|A\|\|B\|\|a_n-x\|\right]\hspace{1in}$$
$$\leq (1-\alpha_n)\|y_{n}-y\|\hspace{4in}$$
$$+\alpha_n \left[\delta_2\theta_2\|y_{n}-y\|+\rho\delta_2\epsilon_2\|x_n-x\| +\lambda\nu_2\gamma\|A\|\|B\|\theta_1\|x_n-x\| \right.\hspace{1in}$$
$$\left. +\lambda \rho\nu_2\epsilon_1\gamma\|A\|\|B\|\|y_n-y\|\right],\hspace{3in}$$ $$~~\leqno(3.12)$$

\noindent where $\delta_2=(1+\gamma \|B\|^2\theta_4)$.

\vspace{.3cm}
Now, define the norm $||.||_{\star}$ on $H_1\times H_2$ by
$$||(x,y)||_{\star} =||x||+||y||,~~~(x,y)\in H_1\times H_2.$$
We can easily show that $(H_1\times H_2, ||.||_{\star})$ is a Banach space.

\vspace{.3cm}
Since $\gamma\|A\|\|B\| <2$ then, using (3.11) and (3.12), we have the following estimate:
$$\|(x_{n+1},y_{n+1})-(x,y)\|_* =\|x_{n+1}-x\|+\|y_{n+1}-y\|\hspace{4in}$$
$$< (1-\alpha_n)[\|x_{n}-x\|+\|y_{n}-y\|]\hspace{.5in}$$
$$\hspace{.5in} +\alpha_n[\delta_1\theta_1+\rho\delta_2\epsilon_2 +2\lambda(\rho\nu_1\epsilon_2+\nu_2\theta_1)]\|x_{n}-x\|$$
$$\hspace{.5in}+ \alpha_n[\delta_2\theta_2+\rho\delta_1\epsilon_1 +2\lambda(\rho\nu_2\epsilon_1+\nu_1\theta_2)]\|y_{n}-y\|$$
$$= [1-\alpha_n(1-\theta)]\|(x_{n},y_{n})-(x,y)\|_*,\hspace{-.2in}\leqno(3.13)$$
where $\theta= \max \{k_1,k_2\};~~k_1=e_1 \theta_1+\rho e_3;~~k_2=e_2 \theta_2+\rho e_4;~~e_1= \delta_1+2\lambda \nu_2;~~e_2= \delta_2+2\lambda \nu_1;~~e_3= \delta_2\epsilon_2+2\lambda\nu_1\epsilon_2;~~e_4= \delta_1\epsilon_1+2\lambda\nu_2\epsilon_1.$

\vspace{0.3cm}
Thus we obtain
 $$\|(x_{n+1},y_{n+1})-(x,y)\|_* < \prod\limits^{n}_{r=1}[1-\alpha_r(1-\theta)] \|(x_{0},y_{0})-(x,y)\|_*. \leqno(3.14)$$

Since $p_i= \frac{e_{i+2}}{e_{i}}$ and $q_i=\frac{1}{e_i}$, it follows from given conditions on $\rho,~\lambda,~\gamma$ that $\theta \in (0,1)$. Since $\sum \limits^{\infty}_{n=1} \alpha_n=+\infty$ and $\theta \in (0,1)$, it implies in the light of \cite{kaz1} that
$$\lim\limits_{n\to \infty}\prod\limits^{n}_{r=1}[1-\alpha_r(1-\theta)]=0.$$

 Thus, it follows from  (3.14) that $\{(x_{n+1},y_{n+1})\}$  converges strongly to $(x,y)$ as $n \to +\infty$, that is, $x_n \to x$ and $y_n \to y$ as $n \to +\infty$. Further, it follows from (3.5) and (3.6), respectively, that $a_n \to x$ and $d_n \to y$  as $n \to +\infty$. Since $A,~B$ are continuous, it follows  that $A(a_n) \to A(x)$ and  $B(d_n) \to B(y)$  as $n \to +\infty$. Hence, it follows from (3.7) and (3.8), respectively, that $b_n \to A(x)$ and  $l_n \to B(y)$  as $n \to +\infty$.   This completes the proof.
\end{proof}
\vspace{.3cm}
Now, we give the following corollaries which are consequences of Theorem {\ref{main}}.

\vspace{.3cm}
If we set $H_3=H_1,~H_4=H_2,~C_3=C_1,~C_4=C_2,~f=F,~g=G,$ and $A=B=I$, identity mapping, then Theorem {\ref{main}} reduces to the following theorem for the convergence analysis of Algorithm 2 for SVIP(1.2)-(1.3).

\vspace{.3cm}
\begin{cor} For each  $s\in\{1,2\},$ let $C_{s}$ be a nonempty, closed and convex subset of real Hilbert space $H_{s}$;  let $F: H_1\times H_2\to H_1$ be $\alpha_{1}$-strongly monotone in the first argument and $(\beta_{1},\epsilon_{1})$-Lipschitz continuous, and
  let $G: H_1\times H_2\to H_2$ be $\alpha_{2}$-strongly monotone in the second argument and $(\epsilon_{2},\beta_{2})$-Lipschitz continuous.  Suppose $(x,y) \in C_1 \times C_2$ is a solution to SVIP(1.2)-(1.3) then the sequence $\{(x_n, y_n)\}$ generated by Algorithm 2  converges strongly to $(x,y)$ provided that  for $i\in \{1,2\}$, the constant $\rho$  satisfy the conditions:
$$0<\rho < \min_{1\leq i \leq 2}\left\{\frac{2(\alpha_i-\epsilon_i)}{\beta_i^2-\epsilon_i^2} \right\}$$
\end{cor}
\vspace{.3cm}
If we set $H_2=H_1,~H_4=H_3,~C_2=C_1,~C_4=C_3,~G=F,~g=f,~B=A,$ and $y=x$,  then Theorem {\ref{main}} reduces to the following corollary for the convergence analysis of Algorithm 3 for  SpVIP(1.10)-(1.11).

\vspace{.3cm}
\begin{cor} For each  $s\in\{1,3\},$ let $C_{s}$ be a nonempty, closed and convex subset of real Hilbert space $H_{s}$;  let $F: H_1\times H_1\to H_1$ be $\alpha_{1}$-strongly monotone in the first argument and $(\beta_{1},\epsilon_{1})$-Lipschitz continuous, and
let $f: H_3\times H_3\to H_3$ be $\sigma_{1}$-strongly monotone in the first argument and $(\mu_{1},\nu_{1})$-Lipschitz continuous. Let $A:H_1 \to H_3$  be bounded linear operator. Suppose $x \in C_1$ is a solution to SSpVIP(1.10)-(1.11) then the sequence $\{x_n\}$ generated by Algorithm 3  converges strongly to $x$ provided that  the constants $\rho, \lambda, \gamma$  satisfy the conditions:
$$\left| \rho- \frac{\alpha_1-\epsilon_1}{\beta_1^2-\epsilon_1^2} \right| <\frac{\sqrt{(\alpha_1-\epsilon_1)^2-(\beta_1^2 -\epsilon_1^2)(1-k^2)}}{\beta_1^2 -\epsilon_1^2};$$
$$\alpha_1>\epsilon_1+\sqrt{(\beta_1^2 -\epsilon_1^2)(1-k^2)};~~\beta_1 > \epsilon_1;~~ k=\frac{1}{\delta_1}<1;$$
$$\delta_1=(1+2\theta_{3});~~\theta_{3}=\sqrt{1-2\lambda\sigma_1+\lambda^2\mu_1^2};~~ \lambda >0;~~ \gamma \in \left(0, \frac{2}{\|A\|^2}\right).$$
\end{cor}

\vspace{0.3cm}
\begin{cor} For each  $s\in\{1,3\},$ let $C_{s}$ be a nonempty, closed and convex subset of real Hilbert space $H_{s}$;  let $F: H_1\times H_1\to H_1$ be $\alpha_{1}$-strongly mixed monotone  and $\beta_{1}$-mixed Lipschitz continuous, and
let $f: H_3\times H_3\to H_3$ be $\sigma_{1}$-strongly mixed monotone and $\mu_{1}$-mixed Lipschitz continuous. Let $A:H_1 \to H_3$  be bounded linear operator. Suppose $ x \in C_1 $ is a solution to SSpVIP(1.10)-(1.11) then the sequence $\{x_n\}$ generated by Algorithm 3  converges strongly to $x$ provided that  the constants $\rho, \lambda, \gamma$  satisfy the conditions:
$$\left| \rho- \frac{\alpha_1}{\beta_1^2} \right| <\frac{\sqrt{\alpha_1^2-\beta_1^2 (1-k^2)}}{\beta_1^2};$$
$$\alpha_1>\beta_1\sqrt{1-k^2};~~ k=\frac{1}{\delta_1}<1;$$
$$\delta_1=(1+2\theta_{3});~~\theta_{3}=\sqrt{1-2\lambda\sigma_1+\lambda^2\mu_1^2};~~ \lambda >0;~~ \gamma \in \left(0, \frac{2}{\|A\|^2}\right).$$

\end{cor}
\vspace{.3cm}
\begin{rem}  It is of further research effort to extend the iterative method and results presented in this paper for the system of split variational inequality problems involving set-valued mappings.
\end{rem}



\bigskip
\bigskip

\end{document}